\documentclass[11pt]{amsart}

\usepackage{amsmath}
\usepackage{amsthm}
\usepackage{hyperref}

\usepackage[margin=1.3in]{geometry}
\numberwithin{equation}{section}

\newtheorem{prop}{Proposition}
\newtheorem{thm}[prop]{Theorem}
\newtheorem{cor}[prop]{Corollary}

\numberwithin{prop}{section}

\theoremstyle{definition}
\newtheorem{defn}[prop]{Definition}

\newtheorem{rmk}[prop]{Remark}

\newcommand{\dt}{\frac{\partial}{\partial t}}
\newcommand{\brs}[1]{\left| #1 \right|}
\newcommand{\gG}{\Gamma}
\newcommand{\gD}{\Delta}
\newcommand{\gd}{\delta}
\newcommand{\gs}{\sigma}

\newcommand{\ga}{\alpha}
\newcommand{\gb}{\beta}
\newcommand{\gl}{\lambda}
\renewcommand{\ge}{\epsilon}
\newcommand{\N}{\nabla}
\renewcommand{\bar}[1]{\overline{#1}}
\newcommand{\del}{\partial}
\newcommand{\delb}{\bar{\partial}}

\newcommand{\bj}{\bar{j}}

\newcommand{\bl}{\bar{l}}
\newcommand{\bn}{\bar{n}}

\newcommand{\til}[1]{\widetilde{#1}}

\DeclareMathOperator{\Sym}{Sym}
\DeclareMathOperator{\Rc}{Rc}

\DeclareMathOperator{\End}{End}
\DeclareMathOperator{\Id}{Id}
\DeclareMathOperator{\tr}{tr}

\begin{document}

\title[Generalized K\"ahler geometry and the
pluriclosed flow]{Generalized K\"ahler geometry and the pluriclosed flow}

\begin{abstract} In \cite{ST1} the authors introduced a parabolic flow for
pluriclosed metrics, referred to as pluriclosed flow. We also demonstrated
in \cite{ST2} that this flow, after certain gauge transformations, gives a
class of solutions to the renormalization group flow of the nonlinear sigma
model with $B$-field.
Using these transformations, we show that our pluriclosed flow preserves
generalized K\"ahler structures in a natural way. Equivalently,
when coupled with a nontrivial evolution equation for the two complex
structures, the $B$-field renormalization group flow also preserves
generalized K\"ahler structure. We emphasize that it is crucial to evolve the
complex structures in the right way to establish this fact.
\end{abstract}

\author{Jeffrey Streets}
\address{Rowland Hall\\
         University of California, Irvine\\
         Irvine, CA 92617}
\email{\href{mailto:jstreets@uci.edu}{jstreets@uci.edu}}

\author{Gang Tian}
\address{Beijing University, China and Princeton University,
         Princeton, NJ 08544}
\email{\href{mailto:tian@math.princeton.edu}{tian@math.princeton.edu}}

\date{\today}

\maketitle

\section{Introduction}

The purpose of this note is to show that the pluriclosed flow
introduced in \cite{ST1} preserves generalized K\"ahler geometry.  This
introductory section introduces the main results in a primarily mathematical
context, while a physical discussion of the results appears in section
\ref{phys}.  First we
recall the concept of a generalized K\"ahler manifold.
\begin{defn} A \emph{generalized K\"ahler manifold} is a Riemannian manifold
$(M^{2n},
g)$
together with two complex structures $J_+, J_-$, each compatible with $g$,
further
satisfying
\begin{gather} \label{GK}
\begin{split}
d_+^c \omega_+ = - d_-^c \omega_- =&\ H,\\
d H =&\ 0.
\end{split}
\end{gather}
\end{defn}
\noindent This concept first arose in the work of Gates, Hull, and Ro\v{c}ek
\cite{GHR}, in their study of $N = (2,2)$ supersymmetric sigma models. Later
these structures were put into the rich context of Hitchin's generalized
geometric structures \cite{Hitchin} in the thesis of Gualtieri
\cite{Gualtieri} (see also \cite{Gualt2}).

Recall that a Hermitian manifold $(M^{2n}, \omega, J)$ is \emph{pluriclosed}
if the K\"ahler form
$\omega$ satisfies $i \del \delb \omega = d d^c \omega = 0$.
Note that a generalized K\"ahler manifold $(M,g,J_+,J_-)$ consists of a pair
of pluriclosed
structures $(M,\omega_+, J_+)$ and $(M, \omega_-, J_-)$ whose associated
metrics are equal and furthermore satisfy the
first equation of (\ref{GK}), where $\omega_{\pm} (\cdot,\cdot) = g
(J_{\pm}\cdot,\cdot)$.
The pluriclosed flow is the time evolution equation
\begin{align} \label{PCF}
\dt \omega =&\ \del \del^*_{\omega} \omega + \delb \delb^*_{\omega} \omega +
\frac{\sqrt{-1}}{2} \del \delb \log \det g.
\end{align}
It follows from Theorem 1.2 in \cite{ST1} that with $\omega_{\pm}$ as initial
metrics,
we get solutions $\omega_{\pm}(t)$ of \eqref{PCF} on $M\times [0,T_{\pm})$.
Let $g_{\pm}(t)$ be the Hermitian metric on $M$
whose K\"ahler form is $ \omega_{\pm}(t)$. As in Theorem 6.5 of \cite{ST2}, we
set
\begin{equation}\label{eq:cfield1}
X_{\pm} \,=\, \left (- J_{\pm} d^*_{g_{\pm}(t)} \omega_{\pm}
(t)\right)^{\sharp_{\pm}},
\end{equation}
where $\sharp_{\pm}$ denotes the natural isomorphism from $T^* M$ onto $TM$
defined using
$g_{\pm}(t)$. Further, let $\phi_{\pm}(t)$ denote the one-parameter family of
diffeomorphisms generated by $X_{\pm}$, with $\phi_{\pm}(0) = \Id$. Then our
main theorem can be stated as follows:

\begin{thm} \label{PCF2BEFD} Let $(M, g, J_+, J_-)$ be a generalized K\"ahler
manifold. With notations as above, one has that
$\phi_+(t)^*(g_+(t)) \,=\, \phi_-(t)^*(g_-(t))$, and we denote this metric
$g(t)$. Furthermore, $T_+ = T_- =: T$, and $(M, g(t), \phi_+(t)^*J_+,
\phi_-(t)^*J_-)$ is
a family of generalized K\"ahler manifolds on $[0, T)$ with initial value
$(M,g,J_+,J_-)$.
\end{thm}

Hence, \eqref{PCF} preserves generalized K\"ahler structures. In \cite{ST2},
we found a striking
relationship between solutions to \eqref{PCF} and the B-field renormalization
group flow, and the proof of Theorem \ref{PCF2BEFD} makes essential use of
this. The B-field renormalization group flow arises from physical
considerations. Consider a pair $(g, H)$ of a
Riemannian metric $g$ and closed three-form $H$ on a manifold $M$. The form
$H$ is thought of as the field strength of a locally defined $2$-form $B$
(i.e. $H = dB$). Given this data, and a dilaton $\Phi$ on $M$, one can
associate a Lagrangian of maps of Riemann surfaces $f: (\Sigma, h) \to M$,
called the worldsheet nonlinear sigma model action, given by
\begin{gather} \label{action}
S = - \frac{1}{2} \int_{\Sigma} \left[ \brs{\N f}^2 +
\frac{\ge^{\ga\gb}}{\sqrt{h}} B_{ij} \del_{\ga} f^i \del_{\gb} f^j - 2 \Phi
R(h) \right] dV_h.
\end{gather}
We have suppressed a scaling parameter $\ga'$ which is often included in this
definition, see (\cite{Polchinski} p. 111) for more detail on this action and
what follows. Imposing cutoff independence of the associated quantum theory
leads to first order renormalization group flow equations 
\begin{gather}
\begin{split} \label{flow}
\dt g_{ij} =&\ - 2 \Rc_{ij} + \frac{1}{2} H_{ipq} H_j^{\ pq}\\
\dt H =&\ \gD_{d} H,
\end{split}
\end{gather}
where $\gD_{d} = - \left( d d^* + d^* d \right)$ is the Laplace-Beltrami
operator.
In general there is a dilaton evolution as well, but this decouples from the
above system after applying a diffeomorphism gauge transformation
(\cite{Woolgar} pg. 6), and so is not directly relevant to the discussion
here.

In view of results in \cite{ST2}, one can ask: {\it Does (\ref{flow}) preserve
generalized K\"ahler geometry}?
As it turns out, in the naive sense in which this question is usually asked,
the answer is
no. Specifically, given $(M^{2n}, g, J_{\pm})$ a generalized K\"ahler
manifold, one may ask whether the solution to (\ref{flow}) with initial
condition $(g, d^c_+ \omega_+)$ remains generalized K\"ahler in the sense that 
$g$ remains compatible with $J_{\pm}$ and the equations (\ref{GK}) hold. This
is false in general. One has to evolve the complex structures appropriately so
that they are compatible with $(M,g(t))$
and consequently, give rise to generalized K\"ahler structures. The
corresponding evolution
equation is not obvious at all, and would be quite difficult to guess directly
from
\eqref{flow}. The key insight comes from the pluriclosed flow and its relation
to \eqref{flow} established in \cite{ST2}.

The next theorem is a reformulation of Theorem \ref{PCF2BEFD} in terms of the
B-field flow.

\begin{thm} \label{mainthm} Let $(M^{2n}, g, J_+, J_-)$ be a generalized
K\"ahler structure. The solution to (\ref{flow}) with initial condition $(g,
d^c_+ \omega_+)$ remains a generalized K\"ahler structure in the following
sense: There exists a parabolic flow of complex structures such that if
$J_{\pm}(t)$ are its solutions with initial value $J_{\pm}$, then
the triple $(g(t), J_+(t), J_-(t))$ satisfies the conditions of (\ref{GK}).
\end{thm}

\noindent In fact, $J_{\pm}(t) = (\phi_t^{\pm})^* J_{\pm}$ for the
one-parameter families of diffeomorphisms $\phi_t^{\pm}$
which relate \eqref{PCF} to \eqref{flow}, as described in Theorem
\ref{PCF2BEFD}.
However, it is unclear yet how to construct $\phi_{\pm}(t)$ from \eqref{flow}
since \eqref{flow} does not tell how to get $J_{\pm}(t)$. A more precise
statement of Theorem \ref{mainthm} is given below as Corollary \ref{maincor}.

We end the paper in section \ref{class} with some structural results that must
be satisfied for pluriclosed structures which evolve under (\ref{flow}) by
homotheties, which we call \emph{static} structures. In particular, we exhibit
some properties showing that a static structure is automatically
K\"ahler-Einstein, and give a complete classification in the case of
non-K\"ahler complex
surfaces.

\begin{thm}\label{classthm} Let $(M^4, g, J)$ be a static pluriclosed
structure and suppose $b_1(M)$ is odd. Then $(M^4, J)$ is locally isometric to
$\mathbb R \times S^3$ with the standard product metric. The universal cover
of $(M, J)$ is biholomorphic to $\mathbb C^2 \setminus \{(0, 0) \}$, and $M$
admits a finite sheeted cover $\til{M}$ with fundamental group $\mathbb Z$,
specifically
\begin{align}
\pi_1(\til{M}) \cong \mathbb Z = \left< (z_1, z_2) \to (\ga z_1, \gb z_2)
\right>
\end{align}
where $\ga, \gb \in \mathbb C$, $1 < \brs{\ga} = \brs{\gb}$.
\end{thm}

\noindent \textbf{Acknowledgements:} The authors would like to thank 
Sergey Cherkis for his comments, and the referee for a careful reading and 
some helpful suggestions.

\section{Physical Interpretation} \label{phys}

The first order RG flow equations (\ref{flow}) are derived by imposing cutoff
independence for the quantum field theory associated to a nonlinear sigma model.
 For the pure gravity model these equations were first derived by Friedan
\cite{Friedan}, yielding the Ricci flow for the order $\ga'$ approximation, while
for the model including a skew-symmetric background field,
these equations were derived in \cite{Friedan2} (see also \cite{Friedan3}).  Recently,
due partly
to the mathematical breakthroughs of Perelman \cite{P1}, this flow has garnered
more interest in the mathematics and physics communities.  In particular, in
\cite{Woolgar} the authors generalized Perelman's $\mathcal F$ functional to
show that (\ref{flow}) is in fact the gradient flow of the lowest eigenvalue of a 
certain Schr\"odinger
operator.  This property is suggested by Zamolodchikov's $c$-theorem \cite{Zal},
which implies the irreversibility of some RG flows.  Furthermore, the first
author showed in \cite{Str} that a certain generalization of Perelman's entropy
functional is monotone for (\ref{flow}).  

In this paper we address a different issue related to the RG flow.  Recall that,
as exhibited in \cite{GHR}, when imposing $N = (2,2)$ supersymmetry, the
equations (\ref{GK}) are induced on the target space of a $2$-dimensional
nonlinear sigma model, whose underlying sigma model action (\ref{action}).  A
very natural question in this context is whether one can expect the
supersymmetry equations to be preserved along the solution to the RG flow, when
away from a fixed point.  Our results show that the system of equations
(\ref{flow}), will \emph{not} in general preserve the $N = (2,2)$ supersymmetry
equations (\ref{GK}).  However, if one
adds an evolution equation for the complex structures $J_{\pm}$, specified in
(\ref{GKflow}), then the renormalized coupling constants $(g(t), H(t),
J_{\pm}(t))$, \emph{will} define a supersymmetric model, for all cutoff scales. 
In fact it is clear from our proofs that the entire discussion is true for the
weaker $N = 2$ supersymmetry equations.  

Our derivation of the evolution equation for $J_{\pm}$ comes from recognizing
special diffeomorphism gauges relating solutions to (\ref{flow}) to solutions of
(\ref{PCF}), where half of the supersymmetry equations (\ref{GK}) are clearly
preserved with respect to a fixed complex structure.  Thus the evolution
equations for $J_{\pm}$ come from the action of the gauge group, hence it is
unlikely one could modify the sigma model action (\ref{action}) to derive these
equations.  This makes equation for $J_{\pm}$ in (\ref{GKflow}) all the more
surprising and mysterious.  Understanding the physical meaning of the evolution
for $J_{\pm}$
therefore remains an interesting open problem.  We also remark that our 
results only apply to the order $\ga'$ approximation of the renormalization group 
flow.  It remains an interesting open problem to ask whether higher order 
approximations, or even the full RG flow, preserve $N = (2,2)$ supersymmetry 
in the sense we have described here.

Finally, Theorem \ref{classthm} can be thought of as a ``No-Go'' theorem for
certain string vacua.  In particular, we have given a complete classification of
supersymmetric solutions to (\ref{flow}) which evolve purely by homothety on
non-K\"ahler surfaces.  In the end only a restricted class of Hopf surfaces can
possibly admit solutions to these equations.  Other structural results on these
vacua in arbitrary dimension appear in section \ref{class}.  An interesting
further problem is to classify solutions to the RG flow which evolve entirely by
the action of the diffeomorphism group.

\section{Proof of main theorems} \label{mainsec}

\begin{proof}[Proof of Theorem \ref{PCF2BEFD}] Consider the Hermitian manifold
$(M^{2n}, g, J_+)$.  By
(\ref{GK}), this is a pluriclosed structure, i.e.
\begin{align}
d d^c_+ \omega_+ = 0.
\end{align}
By (\cite{ST1} Theorem 1.2), there exists a solution to (\ref{PCF}) with
initial condition $\omega_+$ on $[0, T)$ for some maximal $T \leq \infty$.
Call
this one-parameter family of K\"ahler forms $\omega_+(t)$, and define
$\omega_-(t)$ analogously as the solution to (\ref{PCF}) on the complex
manifold $(M, J_-)$ with initial condition $\omega_-$. Next consider the
time-dependent vector fields
\begin{align}
X^{\pm} = \left(- J_{\pm} d^*_{g_{\pm}} \omega_{\pm} \right)^{\sharp_{\pm}},
\end{align}
and let $\phi_{\pm}(t)$ denote the
one-parameter family of diffeomorphisms of $M$ generated by $X^{\pm}$, with
$\phi^{\pm}_0 = \Id$. Theorem 1.2 in \cite{ST2} implies that
$(\phi_+(t)^*g_+(t), \phi_+(t)^*(d^c_+
\omega_+(t)))$ is a solution to (\ref{flow}) with initial condition $(g, d^c_+
\omega_+)$.
Likewise, we have a solution
$(\phi_-(t)^*g_-(t), \phi_-(t)^*(d^c_-
\omega_-(t)))$ to (\ref{flow}) with initial condition $(g, d^c_- \omega_-)$.
However, if we let $(\til{g}(t), \til{H}(t))$ denote this latter solution, we
observe that
\begin{gather}
 \begin{split}
\dt \til{g}_{ij} =&\ - 2 \til{\Rc}_{ij} + \frac{1}{2} \til{H}_{ipq}
\til{H}_j^{\ pq} = - 2 \til{\Rc}_{ij} + \frac{1}{2} \left( - \til{H}_{ipq}
\right)
\left(- \til{H}_j^{\ pq} \right)\\
\dt \left( - \til{H} \right) =&\ \gD_{d} \left( - \til{H} \right),
 \end{split}
\end{gather}
i.e. $(\til{g}(t), - \til{H}(t))$ is a solution to (\ref{flow}) with initial
condition $(g, - d^c_- \omega_-)$. By (\ref{GK}), we see that
$(\phi_+(t)^*g_+(t), \phi_+(t)^*(d^c_+
\omega_+(t)))$
and $(\phi_-(t)^*g_-(t), -\phi_-(t)^*(d^c_-
\omega_-(t)))$ are two solutions of (\ref{flow}) with the same initial
condition. Using the uniqueness of solutions of (\ref{flow}) (\cite{Str}
Proposition 3.3), we conclude that these two solutions
coincide, and call the resulting one-parameter family $(g(t), H(t))$.

Next we want to identify the two complex structures with which $g$ remains
compatible.  We
observe by that for arbitrary vector fields $X$, $Y$,
\begin{gather}
 \begin{split}
g\left( \phi_{\pm}(t)^* J_{\pm} X, \phi_{\pm}(t)^* J Y \right) =&\ g\left(
\phi_{\pm}(t)^{-1}_*\cdot J_{\pm} \cdot\phi_{\pm} (t)_* X, \phi_{\pm}(t)^{-1}
_* \cdot J_{\pm} \cdot \phi_{\pm} (t)_* Y \right)\\
=&\ \left[\phi_{\pm}(t)^{-1,*} g \right] \left(J_{\pm} \cdot
\phi_{\pm}(t)_* X, J_{\pm} \cdot \phi_{\pm}(t)_* Y \right)\\
=&\ g_{\pm} \left(J_{\pm} \cdot \phi_{\pm} (t)_*
X, J_{\pm} \cdot \phi_{\pm} (t)_* Y \right)\\
=&\ g_{\pm} \left(\phi_{\pm} (t)_*
X,  \phi_{\pm} (t)_* Y \right)\\
=&\ \left[\phi_{\pm}(t)^* g_{\pm} \right] (X, Y)\\
=&\ g(X, Y).
 \end{split}
\end{gather}
Therefore $g(t)$ is compatible with $\phi_{\pm}(t)^* J_{\pm}(t)$. Denote
these two time dependent complex
structures by $\til{J}_{\pm}$.  It follows that
$\til{\omega_{\pm}} = \phi_{\pm} (t)^* \omega_{\pm}$. Next we note by
naturality of $d$ that
\begin{gather}
 \begin{split}
\til{d^c_{\pm}} \til{\omega_{\pm}} (X, Y, Z) =&\ - \left[d
\til{\omega_{\pm}}\right] \left(\til{J}_{\pm} X, \til{J}_{\pm} Y,
\til{J}_{\pm}
Z \right)\\
=&\ - \left[ d \phi_{\pm}(t)^* \omega_{\pm} \right]
\left(\phi_{\pm}(t)^{-1}_*\cdot J_{\pm} \cdot \phi_{\pm} (t)_* X,
\cdots
\right) \\
=&\ \left[ \phi_{\pm} (t)^* \left( - d \omega_{\pm} \right)
\right]
\left( \phi_{\pm}(t)^{-1}_* \cdot J_{\pm} \cdot \phi_{\pm} (t)_*X, \cdots
\right)\\
=&\ - d \omega_{\pm} \left( J_{\pm} \cdot \phi_{\pm} (t)_* X, \cdots
\right)\\
=&\ d^c_{\pm} \omega_{\pm} \left( \phi_{\pm}(t) X, \cdots
\right)\\
=&\ \phi_{\pm} (t)^* \left( d^c_{\pm} \omega_{\pm} \right)(X, Y,
Z)\\
=&\ \pm H(X, Y, Z).
 \end{split}
\end{gather}
It follows that
\begin{align}
\til{d^c_{+}} \til{\omega_{+}} =&\ - \til{d^c_{-}} \til{\omega_-} = H, \qquad d
H = 0,
\end{align}
showing that the triple $(g(t), \til{J}_+(t), \til{J}_-(t))$ is generalized
K\"ahler for all
time. This finishes the proof of Theorem \ref{PCF2BEFD}.
\end{proof}

\noindent To prove Theorem \ref{mainthm}, we need to find the evolution
equation for $J_{\pm}(t)$.  Note that our curvature convention is that
$(\N^2_{e_i, e_j} - \N^2_{e_j, e_i} )e_k = R_{ijk}^{\hskip 0.18in l} e_l$.

\begin{prop} \label{Jev} Let $(M^{2n}, \til{g}(t), J)$ be a solution to the
pluriclosed
flow. Let $\phi_t$ be the one parameter family of diffeomorphisms generated
by $\left(- J d^*_{\til{g}} \til{\omega} \right)^{\sharp}$ with $\phi_0 = \Id$,
and let
$g(t) = \phi_t^*(\til{g}(t)), J(t) = \phi_t^*(J)$. Then
\begin{gather} \label{Jflow}
\begin{split}
\dt J_k^l =&\ \left(\gD J\right)_k^l - [J, g^{-1} \Rc]_k^l\\
&\ - J_k^p D^s J_i^l D_p J_s^i - J_i^l D^s J_k^p D_p J_s^i + J_s^p D^s J_i^l
D_p J_k^i + J_i^l D^s J_s^p D_p J_k^i\\
&\ - J_p^l D_k J_t^p D^s J_s^t + J_k^p D_p J_t^l D^s J_s^t - J_t^p D^s J_s^t
D_p J_k^l.
\end{split}
\end{gather}
\begin{proof} It suffices to compute the time derivative of $J$ at $t = 0$.
First we note
\begin{align}
\left(\dt J(t) \right)_{|t = 0} =&\ \dt \left( \phi_t^* J \right)_{|t = 0} =
\mathcal L_{X(0)} J.
\end{align}
One may compute
\begin{align}
( \mathcal L_X J)(Y) = [X, JY] - J[X, Y].
\end{align}
In coordinates this reads
\begin{align}
( \mathcal L_X J)_k^l = J_p^l \del_k X^p - J_k^p \del_p X^l + X^p \del_p
J_k^l.
\end{align}
Furthermore, since the Levi-Civita connection $D$ is torsion-free we have
\begin{align}
( \mathcal L_X J)_k^l = J_p^l D_k X^p - J_k^p D_p X^l + X^p D_p J_k^l.
\end{align}
We next observe a formula for the vector field $X$.
\begin{align}
X^p =&\ - g^{pq} J_q^r (d^* \omega)_r = - J_t^p D^s J_s^t.
\end{align}
Thus
\begin{gather} \label{Jev10}
\begin{split}
\left(\mathcal L_X J \right)_k^l=&\ - J_p^l D_k \left( J_t^p D^s J_s^t
\right)+ J_k^p D_p \left( J_t^l D^s J_s^t \right) - J_t^p D^s J_s^t D_p
J_k^l\\
=&\ D_k D^s J_s^l - J_p^l D_k J_t^p D^s J_s^t + J_k^p J_t^l D_p D^s J_s^t +
J_k^p D_p J_t^l D^s J_s^t\\
&\ - J_t^p D^s J_s^t D_p J_k^l\\
=&\ D^s D_k J_s^l + g^{uv} \left(R_{u k v}^{\hskip 0.18in p} J_p^l - R_{u k
p}^{\hskip 0.18in l} J_v^p
\right)\\
&\ + J_k^p J_t^l D^s D_p J_s^t + J_k^p J_t^l g^{uv} \left( R_{u p v}^{\hskip
0.18in q} J_q^t
-R_{u p q}^{\hskip 0.18in t} J_v^q \right)\\
&\ - J_p^l D_k J_t^p D^s J_s^t + J_k^p D_p J_t^l D^s J_s^t - J_t^p D^s J_s^t
D_p J_k^l.
\end{split}
\end{gather}
We will simplify this expression using the vanishing of the Nijenhuis tensor
of $J$. Recall
\begin{align}
N(X, Y) =&\ [JX, JY] - [X, Y] - J[JX, Y] - J[X, JY].
\end{align}
In coordinates we may express
\begin{align}
N_{jk}^i = J_j^p \del_p J_k^i - J_k^p \del_p J_j^i - J_p^i \del_j J_k^p +
J_p^i \del_k J_j^p.
\end{align}
Again since $D$ is torsion-free we may express
\begin{align}
N_{jk}^i = J_j^p D_p J_k^i - J_k^p D_p J_j^i - J_p^i D_j J_k^p + J_p^i D_k
J_j^p.
\end{align}
Thus since $J$ is integrable we may conclude
\begin{gather}
 \begin{split}
0 =&\ D^k \left( J_i^l N_{jk}^i \right)\\
=&\ D^k \left[J_i^l \left( J_j^p D_p J_k^i - J_k^p D_p J_j^i - J_p^i D_j
J_k^p+ J_p^i D_k J_j^p \right) \right]\\
=&\ D^k D_j J_k^l - D^k D_k J_j^l + D^k \left[J_i^l J_j^p D_p J_k^i - J_i^l
J_k^p D_p J_j^i \right].
 \end{split}
\end{gather}
Plugging this into the first term of (\ref{Jev10}) we conclude
\begin{gather} \label{Jev20}
\begin{split}
\left( \mathcal L_X J \right)_k^l =&\ D^s D_s J_k^l - D^s \left[ J_i^l J_k^p
D_p J_s^i - J_i^l J_s^p D_p J_k^i \right] + J_k^p J_t^l D^s D_p J_s^t\\
&\ + g^{uv} \left(R_{u k v}^{\hskip 0.18in p} J_p^l - R_{u k p}^{\hskip 0.18in
l} J_v^p - J_k^p R_{u p v}^{\hskip 0.18in l} -
J_k^p J_t^l J_v^q R_{u p q}^{\hskip 0.18in t} \right)\\
&\ - J_p^l D_k J_t^p D^s J_s^t + J_k^p D_p J_t^l D^s J_s^t - J_t^p D^s J_s^t
D_p J_k^l\\
=&\ \left(\gD J\right)_k^l + J_i^l J_s^p D^s D_p J_k^i\\
&\ + g^{uv} \left(R_{u k v}^{\hskip 0.18in p} J_p^l - R_{u k p}^{\hskip 0.18in
l} J_v^p - J_k^p R_{u p v}^{\hskip 0.18in l} -
J_k^p J_t^l J_v^q R_{u p q}^{\hskip 0.18in t} \right)\\
&\ - J_k^p D^s J_i^l D_p J_s^i - J_i^l D^s J_k^p D_p J_s^i + J_s^p D^s J_i^l
D_p J_k^i + J_i^l D^s J_s^p D_p J_k^i\\
&\ - J_p^l D_k J_t^p D^s J_s^t + J_k^p D_p J_t^l D^s J_s^t - J_t^p D^s J_s^t
D_p J_k^l.
\end{split}
\end{gather}
Using the skew-symmetry of $J$ one has
\begin{gather} \label{simp1}
\begin{split}
J_i^l J_s^p D^s D_p J_k^i =&\ \frac{1}{2} J_i^l J_s^p \left( D^s D_p - D_p
D^s\right) J_k^i\\
=&\ \frac{1}{2} J_i^l J_s^p g^{st} \left( R_{p t k}^{\hskip 0.18in m} J_m^i -
R_{p t m}^{\hskip 0.18in i}
J_k^m \right)\\
=&\ - \frac{1}{2} g^{st} \left( R_{p t k}^{\hskip 0.18in l} J_s^p + J_i^l J_s^p
J_k^m R_{p t
m}^{\hskip 0.18in i} \right).
\end{split}
\end{gather}
Plugging this into (\ref{Jev20}) yields
\begin{gather}
 \begin{split}
  \dt J_k^l =&\ \left(\gD J\right)_k^l - \frac{1}{2} g^{uv} \left( R_{p v
k}^{\hskip 0.18in l}
J_u^p
+ J_i^l J_u^p J_k^m R_{p v m}^{\hskip 0.18in i} \right)\\
&\ + g^{uv} \left(R_{u k v}^{\hskip 0.18in p} J_p^l - R_{u k p}^{\hskip 0.18in
l} J_v^p - J_k^p R_{u p v}^{\hskip 0.18in l} -
J_k^p J_t^l J_v^q R_{u p q}^{\hskip 0.18in t} \right)\\
&\ - J_k^p D^s J_i^l D_p J_s^i - J_i^l D^s J_k^p D_p J_s^i + J_s^p D^s J_i^l
D_p J_k^i + J_i^l D^s J_s^p D_p J_k^i\\
&\ - J_p^l D_k J_t^p D^s J_s^t + J_k^p D_p J_t^l D^s J_s^t - J_t^p D^s J_s^t
D_p J_k^l.
 \end{split}
\end{gather}
Next we observe the simplification
\begin{gather}
 \begin{split}
  g^{uv} \left( J_i^l J_u^p J_k^m R_{pvm}^{\hskip 0.18in i} + 2 J_k^p J_t^l J_v^q 
R_{upq}^{\hskip 0.18in t} \right) =&\ g^{vu} J_u^p J_i^l J_k^m \left( R_{p v m}^{\hskip 0.18in i}
+ 2 R_{v m p}^{\hskip 0.18in i} \right)\\
=&\ g^{vu} J_u^p J_i^l J_k^m \left( R_{pvm}^{\hskip 0.18in i} + 
R_{v m p}^{\hskip 0.18in i} + R_{m p v}^{\hskip 0.18in i} \right)\\
=&\ 0.
 \end{split}
\end{gather}
Likewise
\begin{gather}
\begin{split}
 g^{uv} \left( R_{pvk}^{\hskip 0.18in l} J_u^p + 2 R_{ukp}^{\hskip 0.18in l} J_v^p 
\right) =&\ g^{uv} J_u^p \left( R_{pvk}^{\hskip 0.18in l} + 2 R_{vkp}^{\hskip 0.18in l} \right)\\
=&\ g^{uv} J_u^p \left( R_{pvk}^{\hskip 0.18in l} + R_{vkp}^{\hskip 0.18in l} 
+ R_{k p v}^{\hskip 0.18in l} \right)\\
=&\ 0
\end{split}
\end{gather}
Finally we note that
\begin{align}
 g^{uv} \left( R_{u k v}^{\hskip 0.18in p} J_p^l - J_k^p R_{u p v}^{\hskip 0.18in l}
 \right) =&\ J_k^p \Rc_p^l - \Rc_k^p J_p^l = [J, g^{-1} \Rc]_k^l
\end{align}
Plugging these simplifications into (\ref{Jev20}) yields the result.
\end{proof}
\end{prop}

With this proposition in hand we can add an equation to the $B$-field flow
system to yield a new system of equations which preserves the generalized
K\"ahler
condition. Specifically, given a Riemannian manifold $(M^n, g)$ and $J \in
\End(TM)$, let
\begin{gather} \label{RQdef}
 \begin{split}
\mathcal R(J)_k^l =&\ [J, g^{-1} \Rc]_k^l\\
\mathcal Q(DJ)_k^l =&\  - J_k^p D^s J_i^l D_p J_s^i - J_i^l D^s J_k^p D_p J_s^i +
J_s^p D^s J_i^l
D_p J_k^i + J_i^l D^s J_s^p D_p J_k^i\\
&\ - J_p^l D_k J_t^p D^s J_s^t + J_k^p D_p J_t^l D^s J_s^t - J_t^p D^s J_s^t
D_p J_k^l.
 \end{split}
\end{gather}
Now consider the system of equations for an a priori
unrelated Riemannian metric $g$, three-form $H$, and tangent bundle
endomorphisms $J_{\pm}$:
\begin{gather} \label{GKflow}
\begin{split}
\dt g_{ij} =&\ - 2 \Rc_{ij} + \frac{1}{2} H_{ipq} H_j^{\ pq}\\
\dt H =&\ \gD_{d} H,\\
\dt J_{\pm} =&\ \gD J_{\pm} + \mathcal R(J_{\pm}) + \mathcal Q(DJ_{\pm}).
\end{split}
\end{gather}

\begin{thm} Let $M^n$ be a smooth compact manifold. Let $g_0 \in \Sym^2(T^*M)$ 
be a Riemannian metric, $H_0 \in \Lambda^3(T^*M)$, $(J_{\pm})_0 \in \End(TM)$.  
There exists $T > 0$ and a unique solution to (\ref{GKflow}) on $[0,
T)$ with initial condition $(g_0, H_0, (J_{\pm})_0)$.
\begin{proof} The proof is by now standard and we only give a sketch. For a
metric $g$ consider the vector field $X_{g}^k = g^{ij} \left( \gG_{ij}^k -
{\left(\gG^0\right)}_{ij}^k \right)$. Now consider the gauge-fixed system
\begin{gather} \label{gfflow}
\begin{split}
\dt g_{ij} =&\ - 2 \Rc_{ij} + \frac{1}{2} H_{ipq} H_j^{\ pq} + \left(L_{X_g}
g\right)_{ij},\\
\dt H =&\ \gD_{d} H + L_{X_g} H,\\
\dt J_{\pm} =&\ \gD J_{\pm} + \mathcal R(J_{\pm}) + \mathcal Q(DJ_{\pm}) +
L_{X_g} J_{\pm}.
\end{split}
\end{gather}
Let $\mathcal O(g, H, J_{\pm})$ denote the total differential operator
representing the right hand sides of (\ref{gfflow}). A calculation shows that
the principal symbol of the linearized operator of $\mathcal O$ is elliptic.
More specifically,
\begin{align}
\left[\gs D \mathcal O \right](\xi)\left(
\begin{matrix}
\gd g\\
\gd H\\
\gd J_{+}\\
\gd J_-
\end{matrix} \right) =
\left(
\begin{matrix}
\brs{\xi}^2 \Id & 0 & 0 & 0\\
\star & \brs{\xi}^2 \Id & 0 & 0\\
\star & 0 & \brs{\xi}^2 \Id & 0\\
\star & 0 & 0 & \brs{\xi}^2 \Id
\end{matrix} \right)
\left(
\begin{matrix}
\gd g\\
\gd H\\
\gd J_{+}\\
\gd J_-
\end{matrix} \right).
\end{align}
It follows from standard results that there is a unique solution to
(\ref{gfflow}) on some maximal time interval $[0, T)$. Moreover, if we let
$\phi_t$ denote the one-parameter family of diffeomorphisms generated by
$-X_t$ satisfying $\phi_0 = \Id$, it follows that $(\phi_t^*(g(t)),
\phi_t^*(H(t)), \phi_t^*(J_{\pm}(t)))$ is a solution to (\ref{GKflow}).
Finally, the proof of uniqueness is the same as that for Ricci flow, where one 
uses that in the modified gauge the diffeomorphisms $\phi_t$ satisfy the
harmonic map heat flow equation. For more detail see (\cite{ChowLu} pg. 117).
\end{proof}
\end{thm}

We now give a corollary to this discussion which is a more precise statement
of Theorem \ref{mainthm}.

\begin{cor}\label{maincor} Let $(M^{2n}, g, J_{\pm})$ be a compact generalized 
K\"ahler manifold. The solution to (\ref{GKflow}) with initial condition $(g,
d^c_{+} \omega_+, J_{\pm})$ remains a generalized K\"ahler structure in that
$g$ is compatible with $J_{\pm}$, $H = \pm d^c_{\pm} \omega_{\pm}$, and
(\ref{GK}) holds at all time the solution exists.
\begin{proof} Let $(g(t), H(t))$ be the solution to (\ref{flow}) with initial
condition $(g, d^c_+ \omega_+)$. We showed in Theorem \ref{mainthm} that
$(g(t), H(t), J_{\pm}(t))$ is a generalized K\"ahler structure for all times,
where $J_{\pm}(t) = \phi_{\pm}(t)^* J_{\pm}$ and $\phi_t^{\pm}$ denote the
one-parameter families of diffeomorphisms used above. However, from
Proposition \ref{Jev} we have that $J_{\pm}$ are solutions of
\begin{align}
\dt J_{\pm} =&\ \gD J_{\pm} + \mathcal R(J_{\pm}) + \mathcal Q(DJ_{\pm}).
\end{align}
It follows that $(g(t), H(t), J_{\pm}(t))$ is the unique solution to
(\ref{GKflow}) with initial condition $(g, d^c_+ \omega_+, J_{\pm})$, and the
result follows.
\end{proof}
\end{cor}

\section{The structure of static metrics} \label{class}

In this section we will collect some results on static pluriclosed solutions
to (\ref{flow}). We begin with some general definitions.

\begin{defn} Let $(M^{2n}, \omega, J)$ be Hermitian manifold with pluriclosed
metric, and let $H = d^c \omega$. We say that $\omega$ is a \emph{$B$-field
flow soliton} if there exists a vector field $X$ and $\gl \in \mathbb R$ such
that
\begin{gather} \label{soliton}
\begin{split}
\Rc - \frac{1}{4} H^2 + L_X g =&\ \gl g,\\
- \frac{1}{2} \gD_{d} H + L_X H =&\ \gl H.
\end{split}
\end{gather}
The form $\omega$ is called \emph{$B$-field flow static}, or simply
\emph{static} for short, if (\ref{soliton}) is satisfied with $X = 0$.
\end{defn}

\begin{prop} \label{staticprop} Let $(M^{2n}, \omega, J)$ be a static
structure. Then
\begin{itemize}
\item{If $\gl = 0$ then $d^* H = 0$ and $b_1(M) \leq {2n}$}
\item{ If $\gl < 0$ then $g$ is K\"ahler, i.e. $H = 0$.}
\item{ If $\gl > 0$ then $\brs{\pi_1(M)} < \infty$.}
\item{ If $\gl \neq 0$ then $[H] = 0$.}
\end{itemize}
\begin{proof}We note that
\begin{align}
\int_M \brs{d^* H}^2 = \int_M \left<d d^* H, H \right> = 2 \gl \int_M
\brs{H}^2.\end{align}
The first part of the first statement and second statement immediately follow.
Now note that
\begin{align}
\Rc = \gl g + \frac{1}{4} H^2.
\end{align}
Since $H^2$ is positive semidefinite, we conclude that if $\gl > 0$ then $\Rc
> 0$. It follows from the Bonnet-Meyers Theorem that $\brs{\pi_1(M)} <
\infty$. Also, if $\gl = 0$, the Bochner argument yields the bound $b_1(M)
\leq 2n$. Finally, since $H$ is closed, if $\gl \neq 0$ we conclude that $H =
\frac{1}{\gl} \gD_{d} H = \frac{1}{\gl} d d^* H$ and hence $[H] = 0$.
\end{proof}
\end{prop}

\begin{cor} \label{constantfix} Let $(M^4, \omega, J)$ be a static structure,
and suppose $M$ is not of K\"ahler type, i.e. $b_1$ is odd. Then $\gl = 0$.
\begin{proof} Since the manifold $M$ does not admit K\"ahler metrics, the
second statement of Proposition \ref{staticprop} rules out $\gl < 0$.
Likewise, if $\gl > 0$ then by the third statement of Proposition
\ref{staticprop} we conclude that $\brs{\pi_1(M)} < \infty$, so that $b_1(M) =
0$, contradicting that $b_1$ is odd. Thus $\gl = 0$.
\end{proof}
\end{cor}

\begin{prop} Let $(M^4, \omega, J)$ be a static structure and suppose $M$ is
of K\"ahler type and $\gl = 0$. Then $\omega$ is K\"ahler-Einstein.
\begin{proof} It follows from the first statement of Proposition
\ref{staticprop} that $H$ is harmonic. It follows from $\del\delb$-lemma that
$H$ is exact, and hence $H$ vanishes and so the metric is K\"ahler-Einstein.
\end{proof}
\end{prop}

\begin{prop} \label{parallelLee} Suppose $(M^4, g, J)$ is a static structure
with $\gl = 0$. Then $D \theta = 0$, that is, the Lee form is parallel with
respect to the Levi-Civita connection.
\begin{proof} As noted above, if $\gl = 0$ then $\Rc = \frac{1}{4} H^2 \geq
0$. Also, we have $d^* H = 0$. But in the case of complex surfaces, $\theta =
\star H$, therefore $\theta$ is harmonic. It then follows from the Bochner
technique that $\theta$ is parallel.
\end{proof}
\end{prop}

Finally we give the proof of Theorem \ref{classthm}.

\begin{proof}[Proof of Theorem \ref{classthm}] It follows from Corollary
\ref{constantfix} that $\gl = 0$. Thus
from Proposition \ref{parallelLee} we conclude that $D \theta = 0$. Recalling
that the pluriclosed flow equations and $B$-field flow equations differ by the
Lie derivative of the vector dual to $\theta$, we conclude that in fact
$\omega$ is static for the pluriclosed flow (see \cite{ST2} Theorem 6.5), i.e.
\begin{align}
\del \del^*_{\omega} \omega + \delb \delb^*_{\omega} \omega +
\frac{\sqrt{-1}}{2} \del \delb \log \det g =&\ 0.
\end{align}
One can check (see \cite{ST1} Proposition 3.3, \cite{IvPap} Proposition 3.3)
that this is the same as
\begin{align}
S - Q =&\ 0
\end{align}
where $S = \tr_{\omega} \Omega$ is the curvature endomorphism associated to
the Chern connection, and
\begin{align}
Q_{i\bj} =& g^{k \bl} g^{m \bn} T_{i k \bn} T_{\bj \bl m},
\end{align}
where $T$ is the torsion of the Chern connection. In the case of surfaces, one
has (\cite{ST1} Lemma 4.4) that $Q = \frac{1}{2} \brs{T}^2 \omega$. Therefore
the metric defines a Hermitian-Einstein connection on $TM$. It follows from
(\cite{GaudIv} Theorem 2) that $(M, g, J)$ is either K\"ahler-Einstein or $g$
is locally isometric, up to homothety, to $\mathbb R \times S^3$. It follows
from \cite{Gauduchon} that the manifold is a Hopf surface. More specifically,
it follows from (\cite{Gauduchon} III Lemma 11) that the fundamental group
takes the form claimed in the theorem.
\end{proof}

\begin{rmk} Indeed each Hopf surface described in Theorem \ref{classthm}
admits static metrics, given by the metric $\frac{1}{\rho^2} \del \delb
\rho^2$, where $\rho = \sqrt{z_1 \bar{z}_1 + z_2 \bar{z}_2}$.
\end{rmk}

\bibliographystyle{hamsplain}

\end{document}